\documentclass{article}

\usepackage{amsmath,amsthm}
\usepackage{amssymb,latexsym}
\usepackage{enumerate}
\usepackage{latexsym}
\usepackage{graphicx}
\usepackage{euscript}
\usepackage{amsfonts}
\usepackage{amssymb}
\usepackage{amsthm}
\usepackage{amsmath}
\usepackage[arrow, matrix, curve] {xy}

\newtheorem{theorem}{Theorem}[section]

\newtheorem{lemma}[theorem]{Lemma}

\newtheorem{proposition}[theorem]{Proposition}

\theoremstyle{definition}

\newtheorem{remark}[theorem]{Remark}

\newtheorem*{xtheorem}{Theorem}

\newcommand{\C}{\mathbb{C}}
\newcommand{\N}{\mathbb{N}}

\newcommand{\T}{\mathbb{T}}
\newcommand{\D}{\mathbb{D}}
\newcommand{\lphi}{\log\varphi}

\newcommand{\Exp}{\textnormal{Exp}}
\newcommand{\Ree}{\textnormal{Re}}

\newcommand{\Bo}{{\cal B}}
\newcommand{\TBo}{{\tilde{\cal B}}}

\newcommand{\HC}{{\cal HC}}
\newcommand{\FHC}{{\cal FHC}}

\newcommand{\tphi}{{\tilde{\varphi}}}

\newcommand{\tK}{{\tilde{K}}}

\newcommand{\Imm}{\textnormal{Im}}

\newcommand{\id}{\textnormal{id}}

\newcommand{\ind}{\textnormal{ind}}
\newcommand{\conv}{\textnormal{conv}}
\newcommand{\len}{\textnormal{len}}

\newcommand{\ldens}{\underline{\textnormal{dens}}}

\newcommand{\linspan}{\textnormal{span}}

\newcommand{\dist}{\mathrm{dist}}

\numberwithin{equation}{section}

\begin{document}


\baselineskip=17pt


\title{Growth of (frequently) hypercyclic functions for differential operators}

\author{Hans-Peter Beise\\
University of Trier\\
FB-IV, Mathematics\\
D-54286 Trier\\
Germany\\
E-mail: beise@uni-trier.de
\and 
J\"{u}rgen M\"{u}ller\\
University of Trier\\
FB-IV, Mathematics\\
D-54286 Trier\\
Germany\\
E-mail: jmueller@uni-trier.de
}
\date{}

\maketitle


\renewcommand{\thefootnote}{}



\renewcommand{\thefootnote}{\arabic{footnote}}
\setcounter{footnote}{0}


\begin{abstract}
We investigate the conjugate indicator diagram or, equivalently, the indicator function of (frequently) hypercyclic functions of exponential type for differential operators. This gives insights into growth conditions of these functions on particular rays or sectors. Our research extends known results in several respects.
\end{abstract}

\section{Introduction}
A continuous operator $T:X\rightarrow X$, with $X$ a topological vector space, is called \textbf{hypercyclic} if there exists a vector $x\in X$ sucht that the orbit $\{T^nx:n\in\N\}$ is dense in $X$. Such a vector $x$ is said to be a \textbf{hypercyclic vector}. By $\HC(T,X)$, we denote the set of all hypercyclic vectors for $T$ (on $X$). The operator is called \textbf{frequently hypercyclic} if there exists some $x\in X$ such that for every non-empty open set $U\subset X$ the set $\{n:T^nx\in U\}$ has positive lower density. The vector $x$ is called a \textbf{frequently hypercyclic vector} in this case and the set of all these vectors shall be denoted by $\FHC(T,X)$ in the following. We recall that the \textbf{lower density} of a discrete set $\Lambda\subset \C$ is defined by
\[
 \liminf\limits_{r\rightarrow \infty}\frac{\#\{\lambda\in\Lambda:|\lambda|\leq r\}}{r}=:\ldens(\Lambda).
\]
We are only concerned with spaces consisting of holomorphic functions and therefore the (frequently) hypercyclic vectors are called (frequently) hypercyclic functions in this work.

In \cite{godefr}, G. Godefroy and J. H. Shapiro show that for every non-constant entire function $\varphi(z)=\sum_{n=0}^\infty c_n z^n$ of exponential type, the induced differential operator
\begin{align}\label{defdifferential}
\varphi(D):H(\C)\rightarrow H(\C), \ f\mapsto \sum_{n=0}^\infty c_n f^{(n)}\; ,
\end{align} 
where $H(\C)$ is endowed with the usual topology of locally uniform convergence, is hypercyclic. This results also applies for the case of frequent hypercyclicity as it is shown in \cite{Thgode}. Actually, in both articles \cite{godefr} and \cite{Thgode}, the outlined results are given for the case of $H(\C^N)$.
The possible rate of growth of the corresponding (frequently) hypercyclic functions is widely investigated (cf. \cite{exphyp}, \cite{frequenterd}, \cite{Thgode}, \cite{hilbertentire}, \cite{duyos}, \cite{erdmannmacl}). It turns out that the level set 
\begin {align}\label{levelset}
C_\varphi:=\{z: |\varphi(z)|=1\}
\end{align}
plays a cruical role in this context. More precisely, under certain additional assumptions, it is known that for $\tau_\varphi:=\dist(0,C_\varphi)$ there are functions of exponential type $\tau_\varphi$ that belong to $\HC(\varphi(D), H(\C))$, while every function of exponential type less than $\tau_\varphi$ cannot belong to $\HC(\varphi(D), H(\C))$ (cf.  \cite{exphyp}). It is also known that for every $\varepsilon>0$ there are functions in $\FHC(\varphi(D), H(\C))$ that are of exponential type less or equal than $\tau_\varphi+\varepsilon$ (cf. \cite{Thgode}). 
For the case of the translation operator $f\mapsto f(\cdot+1)$, which is the differential operator induced by the exponential function, and the ordinary differentiation operator $D$, more precise growth conditions are achieved in  \cite{duyos}, \cite{hilbertentire}, \cite{erdmannmacl} and  \cite{frequenterd}.\\
However, all investigations in this direction have in common that the rate of growth is measured with respect to the \textbf{maximum modulus} $M_f(r):=\max_{|z|=r} |f(z)|$ or $L^p$-averages $M_{f,p}(r):=(1/2\pi\, \int_0^{2\pi}|f(re^{it})|^p dt)^{1/p}$, where $p\in[1,\infty)$. We extend some of these results by considering growth conditions with respect to rays emanating from the origin.\\
For the sake of completeness, we recall that an entire function $f$ is said to be \textbf{of exponential type $\tau$} if 
\[
 \limsup\limits_{r\rightarrow\infty}\frac{\log M_f(r)}{r}=:\tau(f)=\tau,
\]
where we set $\log(0):=-\infty$, and $f$ is said to be a function of exponential type when the above $\limsup$ is not equal to $+\infty$. The \textbf{indicator function} of an entire function of exponential type is defined by
\[
 h_f(\Theta):=\limsup\limits_{r\rightarrow\infty} \frac{\log |f(re^{i\Theta})|}{r},\ \ \Theta\in[-\pi,\pi].
\]
It is known that $h_f$ is determined by the support function of a certain compact and convex set $K(f)\subset\C$, to be more specific, for $z=re^{i\Theta}$ we have
\[
 r\, h_f(\Theta)=H_{K(f)}(z):=\sup\limits\{\Ree(zu):u\in K(f)\}
\]
(cf. \cite{berenstein}). The set $K(f)$ is called the \textbf{conjugate indicator diagram} of $f$. Note that for $f\equiv 0$, we have $K(f)=\emptyset$. In the following we give necessary and sufficient conditions for the location and the size of the conjugate indicator diagram of (frequently) hypercyclic functions for differential operators $\varphi(D)$. According to the above relations, this yields information about the growth on particular rays or sectors in terms of the indicator function. Since 
\[
\max_{\Theta\in[-\pi,\pi]}h_f(\Theta)=\max_{u\in K(f)}|u|=\tau(f),
\]
this also includes information about the possible exponential type. In particular, $f$ is of exponential type zero if and only of $K(f)=\{0\}$.

In the following, we abbreviate the exponential function $z\mapsto e^{\alpha z}$ by $e_\alpha$, for $\alpha$ some complex number. For $\alpha=\tau e^{i\psi}$ the indicator function of $e_\alpha$ is given by 
\[
h_{e_{\alpha}}(\Theta)= \tau \cos(\Theta + \psi)
\] 
and the conjugate indicator diagram is the singleton $\{\alpha\}$. 

With \cite[Theorem 5.4.12]{boas} it follows that for an entire function $f$ of exponential type we have $K(f)=\{\alpha\}$ if and only if there is some entire function $f_0$ of exponential type zero with $f=f_0 e_\alpha$. In that sense, functions which have singleton conjugate indicator diagram are close to the corresponding exponential function. In particular, the indicator functions of $f$ and $e_\alpha$ coincide, which implies that $f$ decreases exponentially in the half plane $|\arg(z) + \psi|> \pi/2$ if $\alpha\not=0$.  

Our first result shows that the conjugate indicator diagram of hypercyclic functions for differential operators are not restricted with respect to their size and shape.
\begin{theorem}\label{satzeins}
Let $\varphi$ be a non-constant entire function of exponential type. Then for every compact and convex set $K\subset \C$ that intersects $C_\varphi$ there exists an $f\in \HC(\varphi(D),H(\C))$ that is of exponential type with $K(f)=K$.
\end{theorem}

Theorem \ref{satzeins} implies that for every $\alpha \in C_\varphi$ there exists some $f_0$ of exponential type zero such that $f=f_0 e_\alpha\in\HC(\varphi(D),H(\C))$. Consequently, in the case that $C_\varphi$ intersects the origin, there is a function $f\in\HC(\varphi(D),H(\C))$ that is of exponential type zero. For the translation operator $e_1(D)$, a much stronger result is due to S. M. Duyos-Ruiz. She proved that functions $f\in \HC(e_1(D), H(\C))$ can have arbitrary slow tranzendental rate of growth, that is, for every $q:[0,\infty)\rightarrow [1,\infty)$ such that $q(r)\rightarrow \infty$ as $r$ tends to infinity, there are functions $f\in \HC(e_1(D),H(\C))$ such that $M_f(r)=O(r^{q(r)})$ (cf.  \cite{duyos}). In \cite{hilbertentire}, this result is extended the Hilbert spaces consisting of entire functions of small growth. 

In Section \ref{s3} we will introduce a transform that quasi-conjugates differential operators and which enables us the extend the result of S. M. Duyos-Ruiz to the whole class of differential operators in the following sense. 

\begin{theorem}\label{satzzwei}
Let $\varphi$ be a non-constant entire function of exponential type and let $\alpha\in C_\varphi$ be so that $\varphi'(\alpha)\neq 0$. Then for every $q:[0,\infty)\rightarrow [1,\infty)$ such that $q(r)\rightarrow \infty$ as $r$ tends to infinity, there is an entire function $f_0$ with $M_{f_0}(r)=O(r^{q(r)})$ and so that $f_0e_\alpha\in \HC(\varphi(D),H(\C))$.
\end{theorem}

The above results fail to hold in the case of frequent hypercyclicity. Here, some expansion of the conjugate indicator diagram is required.

\begin{theorem}\label{satzdrei}
Let $\varphi$ be a non-constant entire function of exponential type.
\begin{enumerate}
\item[(1)] If $K\subset \C$ is a compact and convex set such that the intersection of $K$ and $C_\varphi$ contains a continuum, then there is a function 
$f\in \FHC(\varphi(D),H(\C))$ that is of exponential type and so that $K(f)\subset K$.

\item[(2)] There is no function $f\in \FHC(\varphi(D),H(\C))$ that is of exponential type and so that $K(f)$ is a singleton.

\end{enumerate}

\end{theorem}
In particular, the second part of the above result states that, in contrast to the case of hypercyclicity, a function $f$ of exponential type zero is never frequently hypercyclic for any differential operator $\varphi(D)$ (on $H(\C)$).\\

The proofs of Theorem \ref{satzeins}, \ref{satzzwei} and \ref{satzdrei} will be consequences of results given in the following three sections.

\section{Hypercyclicity of Differential Operators}\label{abschnittzwei}

We start this section with the introduction of a terminology, which will be convenient for what follows: Let $\Omega\subset \C$ be a domain and $K$ a compact subset of $\Omega$. A cycle $\Gamma$ in $\Omega\setminus K$, is called a \textbf{Cauchy cycle for $K$ in $\Omega$} if $\ind_\Gamma(u)=1$ for every $u\in K$ and $\ind_\Gamma(w)=0$ for every $w\in \C\setminus \Omega$. The existence of such a cycle is always guaranteed and, moreover, the Cauchy integral formula 
\[
f(z)=\frac{1}{2\pi i}\int_\Gamma \frac{f(\xi)}{\xi-z}\,d\xi
\]
is valid for every $z\in K$ (see \cite[Theorem 13.5, Theorem 10.35]{rudin}). By $|\Gamma|$ we denote the trace of $\Gamma$ and $\len(\Gamma):=\int_a^b|\Gamma(t)|\,dt$ is the length of $\Gamma$. In the following, $K$ always has simply connected complement. In this case, $\Gamma$ may be chosen as a simple closed path.  

For a given compact and convex set $K\subset \C$, we denote by $\Exp(K)$ the space of all entire functions $f$ of exponential type that satisfy $K(f)\subset K$. This space naturally appears in the context of analytic functionals (cf. \cite{martineau}, \cite{morimoto}, \cite{berenstein}). In the sequel, the differential operators are mainly considered on $\Exp(K)$, which is very convenient as we will see. 

For a function $f$ of exponential type, $\Bo f(z):=\sum_{n=0}^\infty f^{(n)}(0)/z^{n+1}$ is called the \textbf{Borel transform} of $f$. The Borel transform is a holomorphic function on some neighbourhood of infinity that vanishes at infinity. It is known that the conjugate indicator diagram $K(f)$ is the smallest convex, compact set such that $\Bo f$ admits an analytic continuation to $\C\setminus K(f)$ and that the inverse of the Borel transform is given by
\[
f(z)=\frac{1}{2\pi i} \int_\Gamma \Bo f(\xi)\, e^{\xi z}\, d\xi
\]
where $\Gamma$ is a Cauchy cycle for $K(f)$ in $\C$ ( cf. \cite{boas}, \cite{berenstein}). This integral formula is known as the \textbf{P\'{o}lya representation}.

Finally, we  make use of the following notions: $\C_\infty$ is the extended complex plane $\C\cup \{\infty\}$, $\D:=\{z:|z|<1\}$ and $\T:=\{z:|z|=1\}$. If $A\subset \C$, then $A^{-1}:=\{z:1/z\in A\}$, where as usual $1/0:=\infty$, $\overline{A}$ is the closure of $A$ and $\conv(A)$ is the convex hull of $A$. For an open set $\Omega \subset \C_\infty$ the space of functions holomorphic on $\Omega$ and vanishing at $\infty$ (if $\infty \in \Omega$) endowed with the topology of uniform convergence on compact subsets is denoted by $H(\Omega)$. Recall that a function $f$ is said to be holomorphic at infinity if $f(1/z)$ is holomorphic at the origin.

For the proof of the next proposition, we refer to the first chapter of \cite{morimoto}.
\begin{proposition}
Let $K\subset \C$ be a compact and convex set.
\begin{enumerate}
\item[(1)]For every $n\in \N$,
\[
||f||_{K,n}:=\sup\limits_{z\in\C}|f(z)|\,e^{-H_{K}(z)-\frac{1}{n}|z|}
\]
defines a norm $||.||_{K,n}$ on $\Exp(K)$ and the space $\Exp(K)$, endowed with the topology induced by the sequence $\{||.||_{K,n}:n\in\N\}$, is a Fr\'{e}chet space.
\item[(2)] The Borel transfrom
\[
\Bo=\Bo_K:\Exp(K)\rightarrow H(\C_\infty\setminus K), \ f\mapsto \Bo f|_{\C_\infty\setminus K}
\]
is an isomorphism.
\end{enumerate}
\end{proposition}

By differentiation of the parameter integral, the P\'{o}lya representation yields
\[
f^{(n)}(z)=\frac{1}{2\pi i} \int_\Gamma \Bo f(\xi)\, \xi^n \, e^{\xi z} \, d\xi.
\]
Inspired by this formula, we introduce a class of operators on $\Exp(K)$ by replacing $\xi^n$ in the above integral by a function holomorphic on some neighbourhood of $K$. We define $H(K)$ to be the space of germs of holomorphic functions on $K$, where $K\subset\C$ is some compact set.
In order to simplify the notation, we agree the following: An element of $H(K)$ shall always be identified with some of its representatives $\varphi$ which is defined on an open neighbourhood $\Omega_\varphi$ of $K$. In case that $K$ is convex we always assume $\Omega_\varphi$ to be simply connected (actually we may suppose $\Omega_\varphi$ to be even convex). 

Now, for a fixed compact and convex set $K\subset\C$ and a germ $\varphi\in H(K)$, we define
\begin{align}\label{tphidef}
\varphi(D) f(z):=\frac{1}{2\pi i}\int_\Gamma \Bo f (\xi)\,\varphi(\xi)\,e^{\xi z}\, d\xi
\end{align}
where $\Gamma$ is a Cauchy cycle for $K$ in $\Omega_\varphi$. Obviously, this definition is independent of the particular choice of $\Gamma$. If $\varphi$ extends to an entire function $\varphi(z)=\sum_{n=0}^\infty c_n z^n$, the interchange of integration and summation immediately yields
\[
\sum_{n=0}^\infty c_n f^{(n)}(z)=\frac{1}{2\pi i} \int_\Gamma \Bo f(\xi)\, \varphi(\xi) \, e^{\xi z} \, d\xi.
\]
Consequently, we see that the operators $\varphi(D)$ from (\ref{tphidef}) are a natural extension of the differential operators in (\ref{defdifferential}) and this justifies the notation. 

\begin{proposition}\label{phiop}
Let $K$ be a compact, convex set in $\C$ and $\varphi\in H(K)$. Then $\varphi(D)$ defined by (\ref{tphidef}) is a continuous operator on $\Exp(K)$.
\end{proposition}
\begin{proof}
For a given positive integer $n$, we choose $\Gamma$ such that $|\Gamma|\subset \frac{1}{n} \overline{\D}+K$. Then $H_{\conv(|\Gamma|)}\leq H_{K+\frac{1}{n}\overline{\D}}$ and that means $(\Ree(\xi z)- H_K(z)-n^{-1}|z|)\leq 0$ for all $\xi\in |\Gamma|$ and all $z\in\C$. Consequently, 
$
\left|e^{\xi z -H_K(z)-\frac{1}{n} |z|}\right|\leq 1
$  
for all $z\in\C$ and all $\xi\in|\Gamma|$.
As $\Bo:\Exp(K)\rightarrow H(\C_\infty\setminus K)$ is an isomorphism and $|\Gamma|$ is compact in $\C\setminus K$, there is an $m\in\N$ and a constant $C>0$ such that $\sup\{|\Bo f (\xi)|:\xi\in |\Gamma|\}\leq C\left\|f\right\|_{K,m}$. With $M:=\frac{1}{2\pi} \int_\Gamma|\varphi(\xi)|d\xi$, we now obtain
\begin{align*}
\left\|\varphi(D) f\right\|_{K,n}&=\sup\limits_{z\in\C} \left|\frac{1}{2\pi i}\int_\Gamma \varphi(\xi)\,\Bo f(\xi)\,e^{\xi z} \,d\xi\right|e^{-H_K(z)-\frac{1}{n} |z|}
\\[2mm]&\leq\sup\limits_{z\in\C} \frac{1}{2\pi}\int_\Gamma |\varphi(\xi)|\,|\Bo f (\xi)|\left|\,e^{\xi z-H_K(z)-\frac{1}{n} |z|}\right|\, d\xi\leq MC\,\left\|f\right\|_{K,m}.
\end{align*}
This proves that $\varphi(D)$ is a self-mapping on $\Exp(K)$ and the continuity of this operator.
\end{proof}

Now, our main result in this section is as follows:
\begin{theorem}\label{hyperexp}
Let $K$ be a convex, compact subset of $\C$ and $\varphi\in H(K)$ non-constant. Then $\HC(\varphi(D),\Exp(K))\neq\emptyset$ if and only if $\varphi(K)\cap \T\neq \emptyset$.
Further, if $\HC(\varphi(D),\Exp(K))\neq\emptyset$, then the set of all $f \in \HC(\varphi(D),\Exp(K))$ with $K(f)=K$ is residual in $\Exp(K)$ in the sense of Baire categories.
\end{theorem}

Before giving the proof, some auxiliary results for $\Exp(K)$ and $\varphi(D)$ are established.
\begin{proposition}\label{dicht}
Let $K\subset\C$ be a compact and convex set.
\begin{enumerate}
 \item[(1)] For any $\alpha\in K$, the set $\{Pe_\alpha: P \textnormal{ polynomial }\}$ is dense in $\Exp(K)$.
 \item[(2)] If $A$ is an infinite subset of $K$, then $\linspan\{e_\alpha:\alpha\in A\}$ is dense in $\Exp(K)$.
\end{enumerate}
\end{proposition}

\begin{proof}
Let $\Sigma$ denote the space of all polynomials.
In a first case we assume that $0\in K$. For a function $f\in \Exp(K)$ we have that $\TBo f:=(1/\cdot)\,\Bo f(1/\cdot) \in H(\C_\infty\setminus K^{-1})$. Since $\Bo:\Exp(K)\rightarrow H(\C_\infty\setminus K)$ is an isomorphism, one verifies that $\TBo:\Exp(K)\rightarrow H(\C_\infty\setminus K^{-1})$ is also an isomorphism.  
Now, $\Sigma$ is dense in $H(\C_\infty\setminus K^{-1})$ by Runge's theorem and observing that $\TBo^{-1}(\Sigma)=\Sigma$ this shows that $\Sigma$ is dense in $\Exp(K)$.\\
Let $K$ be an arbitrary compact and convex set. By means of \cite[Theorem 5.4.12]{boas} it follows that for every entire function $f$ of exponential type and $\alpha\in\C$ we have $K(fe_{-\alpha})=K(f)-\{\alpha\}$. Thus, if $g=f/e_\alpha$ for an $f\in \Exp(K)$ and $\alpha\in K$,  
\begin{align*}
 ||f||_{K,n}&=\sup_{z\in\C} |g(z)||e^{\alpha z}|e^{-H_K(z)-\frac{1}{n}|z|}\\
 &=\sup_{z\in\C} |g(z)|e^{-H_K(z)-H_{\{-\alpha\}}(z)-\frac{1}{n}|z|} \\&=\sup_{z\in\C} |g(z)|e^{-H_{K-\{\alpha\}}(z)-\frac{1}{n}|z|}\\
 &=||g||_{K-\{\alpha\},n}
\end{align*}
which shows that $f\mapsto f/e_\alpha$ is an isometric isomorphism from $\Exp(K)$ to $\Exp(K-\{\alpha\})$. With the first part, this implies (1).\\
Without loss of generality, we may assume $0\notin A$. It is easily seen that $\Bo e_\alpha=1/(\cdot-\alpha)$ and thus $\Bo(\linspan\{e_\alpha:\alpha\in A\})=\linspan\{1/(\cdot-\alpha):\alpha\in A\}$. Since $A$ has an accumulation point in $K$, a variant of Runge's theorem (see. \cite[Theorem 10.2]{Lubel}) yields that $\linspan\{1/(\cdot-\alpha):\alpha\in A\}$ is dense in $H(\C_\infty\setminus K)$. According to the fact that  $\Bo:\Exp(K)\rightarrow H(\C_\infty\setminus K)$ is an isomorphism, this shows (2).
\end{proof}

A germ $\varphi\in H(K)$ is said to be zero-free if there exists a representative $\varphi$ which is zero-free on some open neighbourhood of $K$. In this case, we always assume that $\Omega_\varphi$ is so small that $\varphi$ is zero-free on $\Omega_\varphi$ and thus $1/\varphi \in H(\Omega_\varphi)$. 

\begin{proposition}\label{iteration}
Let $K\subset\C$ be a compact, convex set and $\varphi,\psi$ in $H(K)$. Then we have $\varphi(D)\psi(D)=\varphi\, \psi(D)$.
In particular, if $\varphi$ is zero-free, then 
\[
\varphi(D)\,(1/\varphi)(D)=(1/\varphi)(D)\,\varphi(D)=\id_{\Exp(K)}
\]
and hence $ \varphi(D)$ is invertible with $\varphi(D)^{-1}=(1/\varphi)(D)$. 
\end{proposition}

Proposition \ref{iteration} is an immediate consequence of  

\begin{lemma}\label{gleichint}
Let $K$ be a compact, convex set in $\C$, $f\in \Exp(K)$ and $\varphi\in H(K)$. Then for all $h\in H(\Omega_\varphi)$, we have
\[
\int_\Gamma \Bo f(\xi)\,\varphi(\xi)\,h(\xi)\,d\xi=\int_\Gamma \Bo (\varphi(D) f) (\xi)\, h(\xi)\,d\xi
\]
where $\Gamma$ is a Cauchy cycle for $K$ in $\Omega_\varphi$.
\end{lemma}
\begin{proof}
Considering Runge's theorem, one verifies that, for fixed $\alpha\in \C$, $\{Pe_\alpha: P \textnormal{ polynomial }\}$ is dense in $H(\Omega)$ whenever $\Omega$ is a simply connected open subset of $\C$. According to the fact that $Pe_\alpha\in \Exp(K)$ for $\alpha\in K$, this shows that $\Exp(K)$ is densely embedded in $H(\Omega)$ for every non-empty, compact and convex $K\subset\C$. Thus, as a consequence of Proposition \ref{dicht} (2), $E:=\linspan\{e_\alpha:\alpha\in\C\}$ is dense in $H(\Omega_\varphi)$ since $\Omega_\varphi$ is simply connected. 

We consider the functional 
\[\left\langle \Lambda,h\right\rangle:=\int_\Gamma \left(\Bo f(\xi)\,\varphi(\xi)-\Bo (\varphi(D) f) (\xi)\right)\,h(\xi)\,d\xi
\]
on $H(\Omega_\varphi)$. With the P\'{o}lya representation for $\varphi(D) f$, the following holds:
\[
\frac{1}{2\pi i}\int_\Gamma \Bo(\varphi(D) f)(\xi)\,e^{\xi \alpha}\,d\xi=\varphi(D) f(\alpha)=\frac{1}{2\pi i}\int_\Gamma \Bo f(\xi)\,\varphi(\xi)\,e^{\xi \alpha}\,d\xi.
\]
Hence $\left\langle \Lambda,e_\alpha\right\rangle=0$ for all $\alpha\in \C$ and consequently $\Lambda|_E=0$. As $E$ is dense in $H(\Omega_\varphi)$, we have $\Lambda=0$. 
\end{proof}

\begin{proposition}\label{gdelta}
Let $K\subset \C$ be a compact, compact set. Then the set of all $f\in\Exp(K)$ with $K(f)=K$ is residual in $\Exp(K)$.
\end{proposition}

\begin{proof}
Let $M\subset H(\C_\infty\setminus K)$ be the set of functions that are exactly holomorphic in $\C_\infty\setminus K$, that means, for every $w\in\C\setminus K$ the radius of convergence of the Taylor series with center $w$ equals $\dist(w,K)$.
Due to a result of V. Nestoridis (see \cite[Theorem 4.5]{nestor}), $M$ is a dense $G_\delta$-set in $H(\C_\infty\setminus K)$. Since $\Bo^{-1}(M)\subset \{f\in\Exp(K):K(f)=K\}$ and $\Bo:\Exp(K)\rightarrow H(\C_\infty\setminus K)$ is an isomorphism, we obtain the assertion.
\end{proof}

Finally, we are prepared for the
\begin{proof}[Proof of Theorem \ref{hyperexp}]
Firstly, assume that $\varphi(K)\subset \D$. Let $\Gamma$ be a Cauchy cycle for $K$ in $\Omega_\varphi$ being so close to $K$ that $|\varphi|<\delta<1$ on $\Gamma$. Then, considering Proposition \ref{iteration}, for any $f\in \Exp(K)$ we have
\[|\varphi(D)^nf(0)|=\left|\frac{1}{2\pi i} \int_\Gamma \Bo f(\xi)\,\varphi^n(\xi)\,d\xi\right|\leq\frac{\delta^n}{2\pi}\int_\Gamma |\Bo f (\xi)|\,d\xi\rightarrow 0
\]
as $n$ tends to infinity. Consequently, $\varphi(D)$ cannot be hypercyclic on $\Exp(K)$. If $\varphi(K)\subset \C\setminus \overline{\D}$, then $\varphi(D)$ is zero-free, as an element of $H(K)$, and thus, by Proposition \ref{iteration}, it is invertible on $\Exp(K)$ with $\varphi(D)^{-1}=(1/\varphi)(D)$. Now, since $(1/\varphi)(K)\subset \D$, we have $\HC((1/\varphi)(D),\Exp(K))=\emptyset$, and this is equivalent to $\HC(\varphi(D),\Exp(K))=\emptyset$ (see \cite[Corollary 1.3]{baymathbook}).\\
Let us now assume that $\varphi(K)\cap\T\neq\emptyset$ and $K$ to have non-emtpy interior. Taking into account that $\varphi$ is non-constant, we have that $\varphi(K)$ has non-empty interior and thus, $\linspan\{e_\alpha:\alpha\in K,\, |\varphi(\alpha)|>1\}$ and $\linspan\{e_\alpha:\alpha\in K,\, |\varphi(\alpha)|<1\}$ are dense in  $\Exp(K)$ by Proposition \ref{dicht} (1). Observing that $\varphi(D)e_\alpha=\varphi(\alpha)e_\alpha$, the Godefroy-Shapiro Criterion (see \cite[Corollary 1.10]{baymathbook}) yields $\HC(\varphi(D),\Exp(K))\neq\emptyset$. In order to prove the hypercyclicity for the case that $K$ has empty interior, we show that $\varphi(D)$ is transitive on $\Exp(K)$ (i.e. for every pair of non-empty open sets $U,V\subset \Exp(K)$ there exists a positive integer $k$ such that $T^k(U)\cap V\neq \emptyset$), which is equivalent to the hypercyclicity of $\varphi(D)$ on $\Exp(K)$ (cf. \cite[Theorem 1.2]{baymathbook}).

It is easily seen that for every positive integer $n$ we find some convex, compact set $K\subset L\subset \Omega$ that has non-empty interior such that $H_L(z)+1/(n+1)|z|<H_K(z)+1/n|z|$ implying $||\cdot||_{L,n+1}<||\cdot||_{K,n}$. Consequently, for given non-empty open sets $U,V\subset\Exp(K)$, we may assume the existence of open sets $\tilde{U},\tilde{V}\subset \Exp(L)$ with $U=\tilde{U}\cap \Exp(K)$ and $V=\tilde{V}\cap\Exp(K)$. By the above, $\varphi(D)$ is hypercyclic and hence transitive on $\Exp(K)$. This implies the existence of a positive integer $k$ such that $\tilde{U}\cap T^{-k}(\tilde{V})$ is a non-empty open set in $\Exp(L)$. The denseness of $\Exp(K)$ in $\Exp(L)$, which is for instance a consequence of Proposition \ref{dicht} (1), yields $\Exp(K)\cap\tilde{U}\cap T^{-k}(\tilde{V})=U\cap T^{-k}(\tilde{V})\neq\emptyset$ so that $T^k(U)\cap V\neq \emptyset$.\\ Since $\Exp(K)$ is a Fr\'{e}chet space, $\HC(\varphi(D),\Exp(K))\neq\emptyset$ implies that $\HC(\varphi(D),\Exp(K))$ is a dense $G_\delta$-set in $\Exp(K)$ (see. \cite[Theorem 2.19]{ErdmannPerisBook}). Due to Proposition \ref{gdelta}, we obtain that $\{f\in\Exp(K):K(f)=K\}\cap\HC(\varphi(D),\Exp(K))$ is residual in $\Exp(K)$.
\end{proof}
As an easy consequence of Theorem \ref{hyperexp} we obtain the
\begin{proof}[Proof of Theorem \ref{satzeins}]
As mentioned in the proof of Lemma \ref{gleichint}, $\Exp(K)$ is densely embedded in $H(\C)$ for every non-empty, compact and convex set $K\subset \C$. Now, if $\varphi$ is an entire function of exponential type, we obtain $\HC(\varphi(D),\Exp(K))\subset \HC(\varphi(D), H(\C))$ and see that Theorem \ref{satzeins} is an immediate consequence of Theorem \ref{hyperexp}.
\end{proof}

\section{(Quasi)-Conjugacy of Differential Operators}\label{s3}

Let $T:X\rightarrow X$ and $S:Y\rightarrow Y$ be two continuous operators acting on topological vector spaces $X,Y$. A very useful tool to link the dynamics of such operators is to show that they are (quasi-) conjugated. That means, find a continuous mapping $\Phi:X\rightarrow Y$ having dense range and such that $\Phi\circ T=S\circ\Phi$, that is, the diagram
\[
\begin{xy}
  \xymatrix{
      X\ar[d]^{\Phi} \ar[rr]^{T}   &     &  X \ar[d]^{\Phi}  \\
      Y \ar[rr]_{S}   &     &  Y
  }
\end{xy}
\] 
commutes. Then $S$ is said to be \textbf{quasi-conjugated} to $T$ (by $\Phi$). If $\Phi$ is bijective and $\Phi^{-1}$ is continuous, then $T$ and $S$ are said to be \textbf{conjugated}.

\begin{proposition}\label{conjugation}
If $S$ is quasi-conjugated to $T$ by $\Phi$, then $\Phi(\HC(T,X))\subset \HC(S,Y)$ and $\Phi(\FHC(T,X))\subset \FHC(S,Y)$. 
\end{proposition}
This result is immediately deduced from the definition of quasi-conjugacy (cf. \cite{ErdmannPerisBook}).\\
In this section, we introduce a transform that quasi-conjugates the operators from Section \ref{abschnittzwei}. Let $K\subset\C$ be a compact and convex set and $\varphi\in H(K)$.
As in the definition of the operators $\varphi(D)$ (cf. \ref{tphidef}),  our starting point is the P\'{o}lya representation. For $f\in\Exp(K)$, we set
\begin{align}\label{defPhi} 
\Phi_\varphi f(z):=\frac{1}{2\pi i}\int_\Gamma \Bo f(\xi)\, e^{\varphi(\xi)z}\, d\xi
\end{align}
where $\Gamma$ is a Cauchy cycle for $K$ in $\Omega_\varphi$. It is clear that this definition is independent of the particular choice of $\Gamma$.

\begin{proposition}\label{densephi}
Let $K$ be a compact, convex subset of $\C$ and $\varphi\in H(K)$ non-constant. Then, for each $f\in \Exp(K)$, the function $\Phi_\varphi f$ defined by (\ref{defPhi}) is an entire function of exponential type with $K(\Phi_\varphi f)\subset\conv(\varphi(K(f)))$. Further 
\[
\Phi_\varphi:\Exp(K)\rightarrow \Exp(\conv(\varphi(K)))
\]
is a continuous operator that has dense range. 
\end{proposition}

\begin{lemma}\label{schnittexp}
Let $K\subset \C$ be a compact, convex set and $(K_n)$ a sequence of compact, convex supersets of $K$ such that $K_n^\circ\supset K_{n+1}$ and $\bigcap_{n\in\N}K_n=K$.
Then $\Exp(K)=\bigcap_{n\in\N}\Exp(K_n)$ in algebraic and topological sense.
\end{lemma}
\begin{proof}
The equality in algebraic sense is clear. That the spaces also coincide in topological sense is an immediate consequenc of the observation that for a given $l\in N$,
we have $H_{K_n,j}\leq H_{K,l}$ for a suitable choice of $n,j\in N$.
\end{proof}

\begin{proof}[Proof of Proposition \ref{densephi}]
One immediately verifies that $\Phi_\varphi f$ is an entire function. We fix some positive integer $n$ and choose $\Gamma$ such that $\varphi(|\Gamma|)$ is contained in  $\conv(\varphi(K(f)))+\frac{1}{n}\overline{\D}$. Then 
\[H_{\conv(\varphi(|\Gamma)|)}(z) \leq H_{\conv(\varphi(K(f)))+\frac{1}{n}\overline{\D}}(z)= H_{\conv(\varphi(K(f)))}(z)+\frac{1}{n}|z|
\]and thus
\begin{align}\label{phiabsch}
||\Phi_\varphi f||_{\conv(\varphi(K(f))),n}&=\sup\limits_{z\in \C}\left|\frac{1}{2\pi i}\int_\Gamma \Bo f(\xi)\, e^{\varphi(\xi)z} d\xi\right|\, e^{-H_{\conv(\varphi(K(f)))}(z)-\frac{1}{n}|z|}\notag \\[2mm]
&\leq \frac{\len(\Gamma)}{2\pi}\,\sup\limits_{\xi\in |\Gamma|} |\Bo f(\xi)| e^{H_{\conv(\varphi(|\Gamma|))}(z)}\,e^{-H_{\conv(\varphi(K(f)))}(z)-\frac{1}{n}|z|}\\[2mm]
&\leq \frac{\len(\Gamma)}{2\pi}\,\sup\limits_{\xi\in |\Gamma|} |\Bo f(\xi)|<\infty.\notag
\end{align}
As $n$ was arbitrary, this yields that $K(\Phi_\varphi f)$ is contained in $\conv(\varphi(K(f)))$, which in particular implies that $\Phi_\varphi f$ is of exponential type and $\Phi f\in \Exp(\conv(\varphi(K)))$.\\
We proceed with the second assertion.
Taking into account that for some $C<\infty$ and $m\in \N$ we have $\sup_{\xi\in |\Gamma|} |\Bo f(\xi)|\leq  C\,||f||_{K,m}$ due to the fact that $\Bo:\Exp(K)\rightarrow H(\C_\infty\setminus K)$ is an isomorphism, the continuity of $\Phi_\varphi$ follows from (\ref{phiabsch}) when $K(f)$ is replaced by $K$. \\
It remains to show that $\Phi_\varphi(\Exp(K))$ is dense in $\Exp(\conv(\varphi(K)))$. 
Therefore, let $K_1,\, K_2, ...$ be a sequence of compact, convex sets in $\Omega_\varphi$ such that $K_n^\circ\supset K_{n+1}$ and the intersection of all these sets is equal to $K$. As already noted above, the Borel transform of $e_\alpha$ is given by $\xi \mapsto 1/(\xi-\alpha)$. Inserting this in (\ref{defPhi}), the Cauchy integral formula yields $\Phi_\varphi(e_\alpha)=e_{\varphi(\alpha)}$ for all $\alpha$ in some $K_n$. Consequently, for arbitrary $n\in\N$ 
\[
\Phi_\varphi(\linspan\{e_\alpha:\alpha\in K_n\})=\linspan\{e_{\varphi(\alpha)}:\alpha\in K_n\}\subset \Exp(\conv(\varphi(K_n)))
\] 
which implies that $\Phi_\varphi:\Exp(K_n)\rightarrow \Exp(\conv(\varphi(K_n)))$ has dense range according to Proposition \ref{dicht}(2) and the fact that $\varphi$ is non-constant.
Since $\Exp(K)$ is dense in $\Exp(K_n)$, we obtain that $\Phi_\varphi(\Exp(K))$ lies densely in $\Exp(\conv(\varphi(K_n)))$.
Furthermore, we have 
\[
\bigcap\limits_{n\in\N}\conv(\varphi(K_n))=\conv(\varphi(K)) 
\]
and hence
\[\bigcap\limits_{n\in\N}\Exp(\conv(\varphi(K_n)))=\Exp(\conv(\varphi(K))).
\]
in algebraic and topological sense by Lemma \ref{schnittexp}. It is now obvious that $\Phi_\varphi(\Exp(K))$ is dense in $\Exp(\conv(\varphi(K)))$.
\end{proof}

In the formulation of Theorem \ref{densephi}, it is necessary to form the convex hull in the image space $\Exp(\conv(\varphi(K)))$, since $\Exp(K)$ is only defined for convex sets $K$. However, we show that the Borel transform of $\Phi_{\varphi}f$ actually admits an analytic continuation beyond $\C_\infty \setminus \conv(\varphi(K))$.
For that purpose, we have to introduce a further notation: For a compact set $K\subset \C$, the polynomially convex hull $\widehat{K}$ is defined as the union of $K$ with the bounded components of its complement.  
Let $K\subset \C$ be a compact, convex set, $f\in \Exp(K)$ and $\varphi\in H(K)$. For $w \in \C \setminus \widehat{\varphi(K)}$ we set
\[
H_\varphi(w):= \frac{1}{2\pi i} \int_{\Gamma} \frac{\Bo f(\xi)}{w-\varphi(\xi)}\,d\xi
\]
with $\Gamma$ a Cauchy cycle for $K\in\Omega_\varphi$ being so near to $K$ that $\varphi(|\Gamma|)$ is contained in a simply connected, compact set $L\supset \widehat{\varphi(K)}$ such that $w\in \C\setminus L$. This definition is independent of the particular choice of $\Gamma$. Since $\varphi(|\Gamma|)$ can be arbitrarily near to $\varphi(K)$, we obtain a function $H_\varphi\in H(\C_\infty\setminus \widehat{\varphi(K)})$.

\begin{proposition}\label{continuation}
The function $H_\varphi \in H(\C_\infty\setminus \widehat{\varphi(K)})$ defines an analytic continuation of $\Bo(\Phi_\varphi) \in H(\C_\infty\setminus\conv(\varphi(K)))$.
\end{proposition}

\begin{proof} Let $\Gamma_0$ be a Cauchy cycle for $\conv(\varphi(K))$ in $\C$. Then we can chose a Cauchy cycle $\Gamma$ for $K$ in $\Omega_\varphi$ being so near to $K$ that 
$\ind_{\Gamma_0}(\varphi(u))=1$ for all $u\in|\Gamma|$. Then
\begin{align*}
\frac{1}{2\pi i} \int_{\Gamma_{0}} H_\varphi(w)  \,e^{w z}\,dw &=\frac{1}{2\pi i} \int_{\Gamma}\Bo f(\xi)\,\frac{1}{2\pi i} \int_{\Gamma_{0}}\frac{e^{w z}}{w-\varphi(\xi)}\, dw \,d\xi\\[2mm]
&=\frac{1}{2\pi i} \int_{\Gamma}\Bo f(\xi)\,e^{\varphi(\xi)z}\,d\xi\\[2mm]
&=\Phi_\varphi f(z)
\end{align*}
by the Cauchy integral formula. Considering that $\Bo_{\conv(\varphi(K))}$ is an isomorphism, we can conclude $H_\varphi|_{\C_\infty\setminus\conv(\varphi(K))}=\Bo(\Phi_\varphi)|_{\C_\infty\setminus\conv(\varphi(K))}$.  
\end{proof}

Now, let $f$ be an entire function of exponential type and $\varphi\in H(K(f))$. Interchanging integration and differentiation yields
\begin{align}\label{differentphi}
(\Phi_\varphi f)^{(n)}(z)=\frac{1}{2\pi i}\int_\Gamma \Bo f(\xi)\,\varphi^n (\xi)\, e^{\varphi(\xi)z}\, d\xi,
\end{align}
which implies that the Taylor expansion of $\Phi_\varphi f$ at the origin is given by
\begin{align}\label{phitaylorinterpolation}
\Phi_\varphi f(z)=\sum\limits_{n=0}^\infty \frac{\varphi(D)^n f(0)}{n!}\, z^n.
\end{align}

Further, in accordance with our conventions, if  $\varphi\in H(K(f))$ is zero-free, $\Omega_\varphi$ is a simply connected domain that contains no zeros of $\varphi$. These conditions ensure the existence of a logarithm function $\lphi\in H(\Omega_\varphi)$ for $\varphi$.  
Then for each non-negative integer $n$, we have
\begin{align}\label{translatephi}
e_1(D)^n\Phi_{\lphi} f(0)&=\Phi_{\lphi} f(n)=\frac{1}{2\pi i}\int_\Gamma \Bo f(\xi)\, e^{n\,\lphi(\xi)} d\xi\notag\\
&=\frac{1}{2\pi i}\int_\Gamma \Bo f(\xi)\, \varphi^n(\xi)\, d\xi=\varphi(D)^n f(0).
\end{align}

We extend (\ref{differentphi}) and (\ref{translatephi}) by showing that $\Phi_\varphi$ commutes with differential operators on $\Exp(K)$. For that purpose, we have to introduce another terminology: A germ $\varphi\in H(K)$ is said to be biholomorphic, if $\Omega_\varphi$ can be choosen so that $\varphi:\Omega_\varphi\rightarrow \varphi(\Omega_\varphi)$ is biholomorphic. In this case, we always assume $\Omega_\varphi$ to be so small that the above property is ensured.

\begin{proposition}\label{phiconj}
Let $K$ be a compact, convex subset of $\C$ and let $\varphi\in H(K)$.  
\begin{enumerate}
\item[(1)]  $D:\Exp(\conv(\varphi(K)))\rightarrow\Exp(\conv(\varphi(K)))$ is quasi conjugated to $\varphi(D):\Exp(K)\rightarrow \Exp(K)$ by $\Phi_\varphi$;
\item[(2)] If $\varphi$ is zero-free then $e_1(D):\Exp(\conv(\log \varphi(K)))\rightarrow \Exp(\conv(\log \varphi(K)))$ is quasi conjugated to $\varphi(D):\Exp(K)\rightarrow \Exp(K)$ by $\Phi_{\log \varphi}$;
\item[(3)] If $C$ is a compact, convex subset of $\C$ and $\psi\in H(C)$ is biholomorphic and satisfies $\psi(C)\supset \varphi(K)$ then $\psi(D):\Exp(\conv(\psi^{-1}\circ\varphi(K)))\rightarrow \Exp(\conv(\psi^{-1}\circ\varphi(K)))$ is quasi conjugated to $\varphi(D):\Exp(K)\rightarrow \Exp(K)$ by $\Phi_{\psi^{-1}\circ\varphi}$.
\end{enumerate}
\end{proposition}
\begin{proof}
Let $f\in\Exp(K)$.
In order to see (1), consider the Taylor expansion for $\Phi_\varphi f$ in (\ref{phitaylorinterpolation}) and observe that \[D\left(\sum\limits_{n=0}^\infty \frac{\varphi(D)^n f(0)}{n!}\, z^n\right)=\sum\limits_{n=0}^\infty \frac{\varphi(D)^{n+1} f(0)}{n!}\, z^n.\] \\
Considering Lemma \ref{gleichint}, we obtain
\begin{align*}
e_1(D)^n\Phi_{\lphi} f(z)&=\frac{1}{2\pi i}\int_\Gamma \Bo f(\xi)\,\,e^{(z+n) \lphi(\xi) } d\xi\\[2mm]
&=\frac{1}{2\pi i}\int_\Gamma \Bo f(\xi)\, \varphi^n(\xi)\,e^{z \lphi(\xi) } d\xi\\[2mm]
&=\frac{1}{2\pi i}\int_\Gamma \Bo \varphi(D)^n f(\xi)\,e^{z\lphi(\xi) } d\xi \\[2mm]
&=\Phi_{\lphi} \varphi(D)^n f(z).
\end{align*}
With $n=1$, this is the assertion in (2). \\
In order to show (3), we consider an arbitrary $z\in \C\setminus C$ and choose a Cauchy cycle $\Gamma_1$ for $K$ in $\Omega_\varphi$ such that $\varphi(|\Gamma_1|)\subset \Omega_{\psi^{-1}}$ and $\psi^{-1}\circ\varphi(|\Gamma_1|)$ is contained in some compact set $L\subset \Omega_\psi$ with $z\in \C\setminus L$. Further, let $\Gamma_2$ be a Cauchy cycle for $L$ in $\Omega_\psi$. Then, according to Proposition \ref{continuation}, we have
\begin{align*}
\psi(D)(\Phi_{\psi^{-1}\circ\varphi}f)(z)&=\frac{1}{2\pi i}\int_{\Gamma_2} \Bo(\Phi_{\psi^{-1}\circ\varphi}f)(w)\, e^{w z}\,dw\\[2mm]
&=\frac{1}{2\pi i}\int_{\Gamma_2} \frac{1}{2\pi i}\int_{\Gamma_1}\frac{\Bo f(\xi)}{w-\psi^{-1}\circ\varphi(\xi)}\,d\xi\, e^{w z}\,dw\\[2mm]
&=\frac{1}{2\pi i}\int_{\Gamma_1} \Bo f(\xi)\,\frac{1}{2\pi i}\int_{\Gamma_2}\frac{\psi(w)\, e^{w z}}{w-\psi^{-1}\circ\varphi(\xi)}\,dw\,\,d\xi\\[2mm]
&=\frac{1}{2\pi i}\int_{\Gamma_1}\Bo f(\xi)\,\varphi(\xi)\, e^{\psi^{-1}\circ\varphi(\xi)z}\,d\xi\\[2mm]
&=\Phi_{\psi^{-1}\circ\varphi}(\varphi(D)f)(z).
\end{align*}
\end{proof}
\begin{remark}
If $\Exp(K)$ is endowed with the relative topology of $H(\C)$, the transform $\Phi_\varphi$ is no longer continuous. Thus, the quasi-conjugacy in Proposition \ref{phiconj} is intimately linked with the topology of $\Exp(K)$.
\end{remark}
As a first application of the introduced transform, we extend the result of Duyos-Ruiz mentioned in the introduction. For that purpose, a further fact has to be used:

In \cite{hilbertentire}, K. C. Chan and J. H. Shapiro strengthened the result of Duyos-Ruiz. Here, growth of entire functions is measured with respect to a so called admissible comparison function, i.e. an entire function $a(z)=\sum_{n=0}^\infty  a_n z^n$ such that $a_n>0$, $a_{n+1}/a_n\rightarrow 0$ as $n\rightarrow \infty$ and $(n+1)\, a_{n+1}/a_n$ is monotonically decreasing. Corresponding to a comparison function $a$, Chan and Shapiro consider 
\[
E^2(a):=\left\{f\in H(\C): ||f||_a^2:=\sum\limits_{n=0}^\infty \frac{|f^{(n)}(0)/n!|^2}{a_n^2} <\infty\right\},
\]
which is a Hilbert space of entire functions. They prove that the translation $e_\alpha(D)$ is hypercyclic on $E^2(a)$ for every admissible comparison function $a$ and every $\alpha\in\C\setminus\{0\}$ (see \cite[Theorem 2.1]{hilbertentire}). In \cite{hilbertentire}, it is also shown that $f\in E^2(a)$ implies $M_f(r)=O(a(r))$.  From the corollary of \cite[Theorem 2.1]{hilbertentire} we can deduce the following:

\begin{xtheorem}[Duyos-Ruiz - Chan and Shapiro]
For every admissible comparison function $a$ and every $\alpha\in\C\setminus\{0\}$, there is an $f\in \HC(e_1(D),\Exp(\{0\}))$ such that $M_f(r)=O(a(r))$.
\end{xtheorem}

By means of the transfrom $\Phi_\varphi$, we show that this result extends to the operators $\varphi(D)$ as follows:

\begin{theorem}\label{growthexp}
Let $K\subset \C$  be a compact, convex set and $\varphi\in H(K)$ non-constant. Then, for every $\alpha\in  K$ such that $|\varphi(\alpha)|=1$, $\varphi'(\alpha)\neq 0$ and every admissible comparison function $a$, there is some $f_0\in \Exp(\{0\})$ that satisfies $M_{f_0}(r)=O(a(r))$ and such that $f=f_0e_\alpha\in \HC(\varphi(D),\Exp(K))$.
\end{theorem}

\begin{lemma}\label{shift}
Let $K\subset\C$ be a compact, convex set, $\varphi\in H(K)$ and $\alpha\in\C$. Then for every $f\in \Exp(K)$, we have $\varphi(D)f=e_\alpha\, \varphi_\alpha(D) (f/e_\alpha)$ where $\varphi_\alpha:=\varphi(\cdot+\alpha)$.

\end{lemma}

\begin{proof}
For $\lambda\in K$, we have $\varphi(D) e_\lambda=\varphi(\lambda)e_\lambda$ and hence 
\[
\varphi(D)e_\lambda=e_\alpha\, \varphi(\lambda)e_{\lambda-\alpha}=e_\alpha\, \varphi(\lambda-\alpha+\alpha)e_{\lambda-\alpha},
\]
which shows the assertion for $f=e_\lambda$, $\lambda\in K$. Since $\varphi$ is holomorphic in a neighbourhood of $K$, we can assume that $K$ has non-empty interior.
Then $\linspan\{e_\lambda:\lambda \in K\}$ is dense in $\Exp(K)$ by Proposition \ref{dicht} (2). Further, as outlined in the proof of Theorem \ref{dicht}, $f\mapsto f/e_\alpha$ is an isometric isomorphism from $\Exp(K)$ to $\Exp(K-\{\alpha\})$ and we can conclude that the above equality extends to all $f\in\Exp(K)$.
\end{proof}

\begin{proof}[Proof of Theorem \ref{growthexp}]
Let $a(z)=\sum_{n=0}^\infty a_n z^n$ be an admissible comparison function. Without loss of generality $a\in \Exp(\{0\})$.
Due to Lemma \ref{shift} we can assume that $\alpha=0$ and thus we only have to show the existence of some $f\in \HC(\varphi(D),\Exp(\{0\}))$ with $M_f(r)=O(a(r))$, $r>0$.
We define $b(z):=\sum_{n=0}^\infty b_n z^n$ with $b_n:=a_n/n!$ which is again an admissible comparison function. Now, as outlined above, the results in \cite{hilbertentire} yield a function 
$g\in E^2(b)\cap \HC(e_1(D),\Exp(\{0\}))$. By the definition of $E^2$ and $(b_n)$, 
\[
\sum_{n=0}^\infty \frac{|g^{(n)}(0)|^2}{(n!\,b_n)^2}=\sum_{n=0}^\infty \frac{|g^{(n)}(0)|^2}{a_n^2}<\infty.
\]
This implies that $G(z):=\sum_{n=0}^\infty |g^{(n)}(0)|z^n\in E^2(a)$ and hence, as again outlined above, $M_G(r)=O(a(r))$.\\
According to the condition $\varphi'(0)\neq 0$, we have that $\varphi$ is biholomorphic as an element of $H(\{0\})$. We can assume that $\varphi(0)=1$, otherwise, replace $e_1$ by $\varphi(0)e_1$ in what follows and notice that $g\in \HC(\varphi(0)e_1(D),\Exp(\{0\}))$ (see \cite[Corollary 3.3]{baymathbook}).
Then $f:=\Phi_{\varphi^{-1}\circ e_1} g\in \HC(\varphi(D), \Exp(\{0\}))$ due to Proposition \ref{conjugation} and Proposition \ref{phiconj}. We find some small $\delta>0$ and $0<c<\infty$ such that $|\varphi^{-1}(e_1(\xi))|\leq c|\xi|$ for all $|\xi|<\delta$. We fix an $r>0$ with $1/r\leq\delta$ and such that for $\Gamma_r:[0,2\pi)\rightarrow \C$, $t\mapsto r^{-1}e^{it}$ we have $e_1(|\Gamma_r|)\subset \Omega_{\varphi^{-1}}$. Now, considering that $\Bo g(\xi)=\sum_{n=0}^\infty g^{(n)}(0)/\xi^{n+1}$ on every compact subset of $\C\setminus \{0\}$, we have
\begin{align*}
M_f(r)&\leq \max_{|z|=r}\left|\frac{1}{2\pi i} \int_{\Gamma_r} \Bo g(\xi)\, e^{(\varphi^{-1}\circ e_1)(\xi) z}\,d\xi\right|\\[2mm]
&\leq  \max_{|z|=r} \sum_{n=0}^\infty |g^{(n)}(0)|\,\left|\frac{1}{2\pi i}\int_{\Gamma_r} \frac{ e^{(\varphi^{-1}\circ e_1)(\xi) z}}{\xi^{n+1}}\, d\xi\right|\\[2mm]
&\leq \sum_{n=0}^\infty |g^{(n)}(0)|\,r^n e^{ \frac{c}{r} r}\\[2mm]
&=e^c\,G(r).
\end{align*}
Thus, $M_f(r)=O(M_G(r))=O(a(r))$ and this completes the proof.
\end{proof}

\begin{proof}[Proof of Theorem \ref{satzzwei}]
According to the proof of Theorem \ref{satzeins}, we have the inclusion $\HC(\varphi(D),\Exp(K))\subset \HC(\varphi(D),H(\C))$ provided that $\varphi(D)$ extends to a continuous operator on $H(\C)$.
Now, the assertion of Theorem \ref{satzzwei} follows from the observation that for each $q(r):[0,\infty)\rightarrow [1,\infty)$ with $q(r)\rightarrow \infty$ as $r\rightarrow \infty$, there exists an admissible comparison function $a$ such that $a(r)=O(r^{q(r)})$ and the application of Theorem \ref{growthexp}.
\end{proof}

\section{Frequent Hypercyclicity of Differential Operators}

In this section we apply $\Phi_\varphi$ to extend known results for frequently hypercyclic functions for $e_1(D)$ to the whole class of differential operator $\varphi(D)$ on $\Exp(K)$ as well as on $H(\C)$. \\

In \cite{BeiseTrans}, the first author proves the following

\begin{xtheorem}
If $K\subset \C$ is a compact, convex set that contains two distinct points of the imaginary axis, then $\FHC(e_1(D),\Exp(K))\neq \emptyset$. 
\end{xtheorem}

We can conclude that it is sufficient to require that $e_1(K)\cap \T$ contains a continuum in order to have $\FHC(e_1(D),\Exp(K))\neq \emptyset$. Similarily, this result holds in the general situation:

\begin{theorem}\label{frequentphi}
Let $K\subset\C$ be a compact, convex set and $\varphi\in H(K)$ non-constant such that $\varphi(K)\cap \T$ contains a continuum. Then we have $\FHC(\varphi(D),\Exp(K)))\neq \emptyset$.
\end{theorem}

\begin{proof}
Our assumptions ensure the existence of a compact, convex set $\tK\subset K$ such that $\varphi(\tK)$ contains some continuum of $\T$ and $\varphi$ is biholomorphic as an element of $H(\tK)$. We choose suitable real numbers $a<b$ so that $e^{[ia,ib]}\subset\varphi(\tK)$. The preceding result yields an $f\in \FHC(e_1(D),\Exp([ia,ib]))$, and, by   Proposition \ref{conjugation} and Proposition \ref{phiconj}, we have 
\[\Phi_{\varphi^{-1}\circ e_1} f\in \FHC(\varphi(D),\Exp(\tK))\subset\FHC(\varphi(D),\Exp(K)).
\]  
\end{proof}

Our next result shows that to some extent the assumption in Theorem \ref{frequentphi} are sharp.

\begin{theorem}
Let $\lambda$ be a complex number and $\varphi\in H(\{\lambda\})$. Then we have $\FHC(\varphi(D),\Exp(\{\lambda\}))=\emptyset$.
\end{theorem}

\begin{proof}
If there exists some $f\in \FHC(\varphi(D),\Exp(\{\lambda\})$, then,  by Proposition \ref{conjugation} and Proposition \ref{phiconj}, 
\[\Phi_\varphi f\in \FHC(D,\Exp(\{\varphi(\lambda)\}))\subset \FHC(D,H(\C)),
\]
contradicting Theorem \ref{satzdrei}(2). 
\end{proof}

Theorem \ref{satzdrei}(2) is stronger than the previous result since it excludes frequent hypercyclicity with respect to the weaker topology of $H(\C)$.  Unfortunately, the transform $\Phi_\varphi$ does not carry over (frequent) hypercyclicity with respect to this topology. Thus, some extra argument is required to show Theorem \ref{satzdrei}(2).

\begin{proof}[Proof of Theorem \ref{satzdrei}]
The first part is an immediate consequence of Theorem \ref{frequentphi} since $\FHC(\varphi(D),\Exp(K))\subset \FHC(\varphi(D), H(\C))$. Thus, there is only (2) left to prove.\\
We suppose there is some entire function $f$ of exponential type such that $K(f)=\{\lambda\}$, $\lambda\in \C$ and $f\in \FHC(\varphi(D),H(\C))$. Then necessarily, $|\varphi(\lambda)|\geq 1$ because otherwise $\varphi(D)^nf(0)\rightarrow 0$ as $n\rightarrow \infty$ as it turns out from the proof of Theorem \ref{hyperexp}, and $\varphi$ is non-constant. 
Hence in some sufficiently small and simply connected, open neighbourhood $\Omega$ of $\lambda$, the function $\tphi:=\varphi/\varphi(\lambda)$ is zero-free, which implies the existence of a logarithm function $\log \tphi$ for $\tphi$ on $\Omega$ with $\log \tphi(\lambda)=0$. We set $h:=\Phi_{\log\tphi} f$.
Then $K(h)=\{0\}$ by Proposition \ref{densephi} and, according to (\ref{translatephi}) applied to $\tphi$, we have 
\begin{align}\label{eqeins}
h(n)=\frac{1}{\varphi(\lambda)^{n}}\,\varphi(D)^n f(0) \textnormal{   for all  } n\in\N\cup\{0\}.
\end{align}
Let $S$ be the sector $\{z:|\arg(z)|\leq \frac{\pi}{5}\}\setminus\{0\}$. By the Casorati-Weierstrass theorem, we can choose $\alpha\in\C$ such that $\varphi(\alpha)$ is close enough to $\pi\varphi(\lambda)$ to ensure that 
\begin{align}\label{eqzwei}
\frac{\varphi(\alpha)}{\varphi(\lambda)}\,S \subset \left\{z:|\arg(z)-\pi|\leq\frac{\pi}{4}\right\}
\end{align}
and $\varphi(\alpha)\neq 0$.
Now, according to the continuity of $\varphi(D)$ on $H(\C)$, for every $\varepsilon>0$, there are some $r>0$ and $\delta>0$ such that for all $g\in H(\C)$ that satisfy
\begin{align}\label{eqdrei}
\sup\limits_{z\in r\,\overline{\D}} |g(z)-e_{\alpha}(z)|<\delta,
\end{align}
we have 
\[
|\varphi(D)g(0)-\varphi(D)e_{\alpha}(0)|=|\varphi(D)g(0)-\varphi(\alpha)|<\varepsilon.
\]  
We assume that $\delta,\varepsilon>0$ are so small that, whenever $g$ satisfies (\ref{eqdrei}), we have
\begin{align}\label{eqvier}
g(0)\in S \textnormal{   and   } \varphi(D)g(0)\in \varphi(\alpha)\,S.
\end{align}
Our assumption implies the existence of some sequence $(n_k)_{k\in \N}$ of positive integers with $\ldens((n_k)_{k\in\N})>0$ and such that $\sup\limits_{z\in r\,\overline{\D}} |\varphi(D)^{n_k}f(z)-e_{\alpha}(z)|<\delta$ for all $k\in \N$. The interpolating property of $h$ in (\ref{eqeins}) combined with (\ref{eqvier}) yields
\begin{align}\label{eqfunf}
h(n_k) \in \frac{1}{\varphi(\lambda)^{n_k}} \,S \textnormal{  and   } h(n_k+1)\in \frac{\varphi(\alpha)}{\varphi(\lambda)^{n_k+1}} \,S \textnormal{  for all  } k\in \N.
\end{align}
Condition (\ref{eqzwei}) implies that the factor $\frac{\varphi(\alpha)}{\varphi(\lambda)}$ rotates $S$ by an angle larger than $\frac{\pi}{2}$. Hence, from (\ref{eqfunf}), it follows that for each $k\in\N$ either $\Ree(h)$ or $\Imm(h)$ has a sign change in $[n_k,n_k+1]$. The intermediate value theorem yields a sequence $(w_k)_{k\in\N}$ of positive numbers with $w_k\in (n_k, n_k+1)$ and 
\begin{align}\label{eqsechs}
\Ree(h(w_k))\,\Imm(h(w_k))=0 \textnormal{ for all } k\in \N.
\end{align}
Assuming that the Taylor series of $h$ is given by $\sum_{\nu=0}^\infty \frac{h_\nu}{\nu!}z^\nu$, we set $h_1(z):=\sum_{\nu=0}^\infty \frac{\Ree(h_\nu)}{\nu!}z^\nu$ and $h_2(z):=\sum_{\nu=0}^\infty \frac{\Imm(h_\nu)}{\nu!}z^\nu$. The functions $h_1, h_2$ are of exponential type zero due to the fact that $h$ is of exponential type and thus $h_1 h_2$ is a function of exponential type zero. Since $\Ree(h(x))=h_1(x)$ and $\Imm(h(x))=h_2(x)$ for every real $x$, we obtain $h_1 h_2(w_k)=0$ for all $k\in \N$ by (\ref{eqsechs}). Taking into account that $(w_k)_{k\in\N}$ has obviously the same lower density as $(n_k)_{k\in\N}$, we have that $h_1h_2$ is a function of exponential type zero having zeros of positive lower density which is impossible unless it is constantly zero (cf. \cite[Theorem 2.5.13]{boas}). 
\end{proof}

\begin{remark}
Let $\varphi$ be an entire function of exponential type. Then for every entire function $f$ of exponential type satisfying $\varphi(K(f))\subset \D$, we have $\varphi(D)^n f(0)\rightarrow 0$, $n\rightarrow \infty$, as it turns out from the proof of Theorem \ref{hyperexp}. Therefore, 
$\HC(\varphi(D),H(\C))\cap\Exp(K)$ (and thus $\FHC(\varphi(D),H(\C))\cap \Exp(K)$)
are empty 
whenever $\varphi(K)\subset \D$. Moreover, Theorem \ref{hyperexp} states that 
$\HC(\varphi(D),\Exp(K))$ (and thus $\FHC(\varphi(D),\Exp(K)$) are empty  
if $\varphi(K)\subset \C\setminus\overline{\D}$. However, we do not know whether $\HC(\varphi(D),H(\C))\cap\Exp(K)$ or $\FHC(\varphi(D),H(\C))\cap\Exp(K)$ are empty in the latter case.
\end{remark}

\subsection*{Acknowledgements}
We are grateful to the DFG (Deutsche Forschungsgemeinschaft; GZ: MU 1529/5-1) for funding the project from which this work emerged.


\begin{thebibliography}{99}
%
\bibitem{baymathbook} F. Bayart and \'E. Matheron. \textit{Dynamics fo Linear Operators.} Cambridge University Press, Cambridge, 2009.
\bibitem{BeiseTrans} H.-P. Beise. Frequently Birkhoff-universal functions. \textit{Submitted}
\bibitem{berenstein} C. A. Berenstein und R. Gay. \textit{Complex analysis and special topics in harmonic analysis.} Springer, New York, 1995.
\bibitem{exphyp} L. Bernal-Gonz\'alez and A. Bonilla. Exponential type of hypercyclic entire functions. \textit{Arch. Math.,} 78: 283-290, 2002.
\bibitem{frequenterd} O. Blasco, A. Bonilla, and K.-G. Gro{\ss}e-Erdmann. Rate of growth of frequently hypercyclic functions. \textit{Proc. Edinb. Math. Soc.,} 53: 39-50, 2010.
\bibitem{boas} R. P. Boas, \textit{Entire functions.} Academic Press, New York, 1954.
\bibitem{Thgode} A.Bonilla and K.-G. Gro{\ss}e-Erdmann. On a theorem of Godefroy and Shapiro. \textit{Integral Equations Operator Theory, 56: 151-162, 2006.}
\bibitem{hilbertentire} K. C. Chan and J. H. Shapiro. The cyclic behaviour of translation operators on Hilbert spaces of entire functions. \textit{Indiana Univ. Math. J.,} 40: 1421-1449, 1991.
\bibitem{duyos} S. M. Duyos-Ruiz. On the existence of universal functions. \textit{Soviet Math. Dokl.,} 27: 9-13, 1983.
\bibitem{godefr} G. Godefroy and J. H. Shapiro. Operators with dense, invariant cyclic vector manifolds. \textit{J. Funct. Anal.,} 98: 229-269, 1991.
\bibitem{erdmannmacl} K.-G. Gro{\ss}e-Erdmann. On the universal functions of G. R. MacLane. \textit{Complex Variables Theory Appl.,} 15: 193-196, 1990.
\bibitem{ErdmannPerisBook} K.-G. Gro{\ss}e-Erdmann and A. Peris. \textit{Linear Chaos.} Springer, Berlin, 2011.
\bibitem{Lubel}D. H. Luecking and L. A. Rubel. \textit{Complex analysis, a functional analysis approach.} Springer, New York, 1984.
\bibitem{martineau} A. Martineau. Sur les fonctionelles analytiques et la transformation de Fourier-Borel. \textit{J. Anal. Math.}, 11: 1-163, 1963.
\bibitem{morimoto} M. Morimoto. \textit{An introduction to Sato's hyperfunctions.} Trans. Math. Monogr. Amer. Math. Soc., 1994.
\bibitem{nestor} V. Nestoridis. Non extendable holomorphic functions. \textit{Math. Proc. Cambridge Philos. Soc.,} 139: 351-360, 2005.
\bibitem{rudin} W. Rudin. \textit{Real and complex analysis.} Mc-Graw-Hill, New York., 3. edition, 1974. 

\end{thebibliography}

%

\end{document}